\documentclass[12pt]{amsart} 

\usepackage[utf8]{inputenc}
\usepackage{amsmath, amsthm, amsfonts, amssymb}
\usepackage{amsaddr}
\usepackage{hyperref} 
\usepackage{xcolor}
\usepackage{float}
\usepackage[shortlabels]{enumitem}
\usepackage{tikz}
\usepackage{diagbox}

\usepackage{pgfplots}
\pgfplotsset{compat=1.15}
\usepackage{mathrsfs}
\usetikzlibrary{arrows}
\usetikzlibrary[patterns]

\definecolor{light-gray}{gray}{0.95}

\usepackage[normalem]{ulem}

\newtheoremstyle{morespace}
  {\baselineskip}   
  {\baselineskip}   
  {\itshape}  
  {0pt}       
  {\bfseries} 
  {.}         
  {5pt plus 1pt minus 1pt} 
  {}          

\theoremstyle{morespace}

\newtheorem{theorem}{Theorem}

\theoremstyle{definition}
\newtheorem{remark}[theorem]{Remark}
\newtheorem{example}[theorem]{Example}

\newcommand{\eps}{\varepsilon}
\renewcommand{\phi}{\varphi}
\newcommand{\e}{\varepsilon}

\newcommand{\N}{\mathbb{N}}

\newcommand{\R}{\mathbb{R}}

\newcommand{\be}{\begin{equation}}
\newcommand{\ee}{\end{equation}}
\newcommand{\co}{\colon\, }

\newcommand{\cont}{\mathrm{con}} 
\newcommand{\arb}{\mathrm{arb}} 
\newcommand{\lin}{\mathrm{lin}} 
\newcommand{\Lip}{\mathrm{lip}} 

\newcommand{\dF}{d_F}
\newcommand{\dG}{d_G}


\title[Noisy information and entropy numbers]{
Noisy nonlinear information   \\
and entropy numbers 
}

\author{
David Krieg, Erich Novak, Leszek Plaskota, \\ and Mario Ullrich
}
\date{\today}

\begin{document}

\begin{abstract}
It is impossible to recover a vector from $\R^m$ 
with less than $m$ linear measurements,
even if the measurements are chosen adaptively.
Recently, it 
has been shown that one can recover vectors from $\R^m$ with arbitrary precision 
using only 
$O(\log m)$ continuous
(even Lipschitz) adaptive measurements,
resulting in an exponential speed-up of continuous information compared to linear information
for various approximation problems.
In this note, 
we characterize the quality of 
optimal \mbox{(dis-)}continuous information 
that is disturbed by deterministic noise 
in terms of entropy numbers. 
This shows that in the presence of noise the potential gain of 
continuous over linear measurements is 
limited, but significant in some cases.

\end{abstract}

\maketitle


\section{Preliminaries}
\label{sec:intro}

Hans Triebel and his co-authors 
used entropy numbers to study the distribution of eigenvalues of differential operators, 
see~\cite{CaTr80,EdTr96,HaTr94}. 
In this note we show that entropy numbers
can also be used to characterize the quality of optimal approximations of operators by algorithms relying 
on information disturbed by determinstic noise.
Our main motivation is the recent paper~\cite{KNU25} 
on noise-free continuous information,
where it is shown that adaption can lead to an exponential speed-up
both compared to non-adaptive continuous information and compared to adaptive linear information. 
It turns out that something similar is not possible with noise, 
but there can still be advantages of nonlinear adaptive information.

Specifically, we 
study minimal worst-case errors for approximating 
a~mapping $$S\co F\to G$$ 
for some set $F$ and a metric space $G$, that can be achieved with \emph{algorithms} 
that have only access to a limited number $n$ of 
measurements. For us, and as it is typical in information-based complexity \cite{NW1,Pl96,TWW88}, 
an algorithm consists of some 
\emph{information mappings} (aka~measurements) 
and a \emph{reconstruction mapping},  
where we put special emphasis on the case where the information is subject to  
some deterministic \emph{noise}. 

That is, a mapping 
$A_n\colon F\to G$ is an algorithm for approximating $S$ if it is of the form 
$A_n(f)=\varphi(\mathbf y)$, where 
$\varphi \colon \mathbb R^n\to G$ is an arbitrary 
(reconstruction) mapping and 
$\mathbf{y}=(y_1,y_2,\ldots,y_n)\in\R^n$ 
is 
information about $f$.  
It satisfies 
\begin{equation}\label{eq:info}|y_i-\lambda_i(f;y_1,\ldots,y_{i-1})|\le\delta,\qquad  1 \le i\le n,
\end{equation}
for some functionals 
\begin{equation}\label{eq:funals}\lambda_i(\,\cdot\,;y_1,\ldots,y_{i-1})\colon F \to [-1,1]
\end{equation}
that belong to a given class $\Lambda.$
To avoid trivial cases, we will also assume that $\delta<1$. 
We notice that the restriction of the range of $\lambda_i$ to $[-1,1]$ (or to any other compact interval)
is needed to make 
the results of this paper meaningful, cf. Remark~\ref{rem0}.
Also note that the noise is deterministic and independent over the different measurements, 
which is sometimes called \emph{adversarial}. 
One may think of $\delta$ as the machine precision or 
other limitation of the \emph{measuring device}, and 
$\lambda_i$ as 
the \emph{ideal} measurement.

Observe that the choice of the successive functionals (measurements) 
$\lambda_i$ depends on the 
previously obtained information $y_1,\ldots,y_{i-1}$ and 
is therefore called \emph{adaptive}. 
Moreover, by~\eqref{eq:info}, we assume that the
results of the measurements $\lambda_i$ are only known up to some error~$\delta$ in the form of a vector $\mathbf{y}$ that 
we call 
\emph{adaptive noisy information} about $f$. 
We also distinguish the subclass of \emph{nonadaptive} 
information, in which case the $\lambda_i$'s do not 
depend on $y_1,\ldots,y_{i-1}.$  

Hence, an algorithm $A_n$ is specified by $\phi$ and (a rule to choose) the $\lambda_i.$ 
Its \emph{(worst case) error for noise bounded by $\delta$} is defined by 
$$
e(A_n,\delta)\,:=\,\sup\Bigl\{\dG(S(f),\varphi(\mathbf y))\mid\,f\in F,\;\mathbf{y}\,
\text{ satisfies}~\eqref{eq:info}\Bigr\},
$$
where $\dG$ is the metric of $G.$
We allow free choice of $\phi$ and $\lambda_i,$ 
and we study how the minimal error of 
algorithms $A_n$, defined as
$$ e_n(S,\delta)\,:=\,\inf_{A_n}\, e(A_n,\delta), $$
depends on $n$, $\delta,$ and 
the class $\Lambda$ of functionals.
We distinguish the classes of \emph{all}
arbitrary (including discontinuous), continuous, Lipschitz (with fixed constant $L$), and linear functionals, and denote the respective minimal errors as
$$e_n^\arb(S,\delta),\quad e_n^\cont(S,\delta),\quad e_{n,L}^\Lip(S,\delta),\quad e_n^\lin(S,\delta).$$ 

In case of continuous and Lipschitz functionals we obviously assume that $F$ is equipped with a metric $d_F,$ and in case of linear functionals we assume that $F$ is a set in a linear space. We refer to~\cite{KNU24} for a recent exposition on adaption (and randomization) for noise-free linear measurements.

\begin{remark}\label{rem0}
If we extend the range of functionals from $[-1,1]$ to $\R$ or do not impose any restriction on the Lipschitz constant (in case of Lipschitz functionals), then we find ourselves in a rather unwanted situation where the value of any functional $\lambda$ can be obtained with arbitrary precision $\delta_1>0.$ Indeed, letting $\lambda_1=\eta\lambda$ with $\eta=\delta/\delta_1$ and measuring $\lambda_1(f)$ one gets $y_1$ satisfying $|y_1-\lambda_1(f)|\le\delta,$ and this means that $y=y_1/\eta$ satisfies $|y-\lambda(f)|\le\delta_1.$ This simple observation leads to a conclusion that then $e_n(S,\delta)=\lim_{\delta_1\to 0^+} e_n(S,\delta_1)$ for any $\delta>0.$
 
As an illustration, we show that if $G$ is a separable metric space then one can construct an approximation with arbitrarily small error $\varepsilon>0$ using just one noisy discontinuous measurement with accuracy $\delta<1$. Indeed, we can select a sequence $\{g_j\}_{j=1}^\infty$ that is dense in $G,$ and let $\lambda(f)=2\cdot k(f)$, where $k(f)$ is such that $
\dG(S(f),g_{k(f)})\le\varepsilon$.
Then, $k(f)$ can be uniquely determined from information $y$ satisfying $|y-\lambda(f)|\le\delta.$ The approximation $A_1(f)=g_{k(f)}$ is well defined and has error $\varepsilon.$ 

Finally, observe that if $F$ is the unit ball in a normed space then \eqref{eq:funals} is equivalent to the quite natural assumption that the allowable linear functionals are exactly those whose norm is at most $1.$
\end{remark}

\section{Main results}\label{sec:main}

There is a characterization of the minimal error $e_n^\arb(S,\delta)$ for arbitrary measurements in terms of the entropy numbers of the mapping $S\co F\to G$. For $n\in \N$, the $n$-th entropy number of $S$ is defined by 
\[
 \e_n(S) \,:=\, \inf\{ \varepsilon>0 \mid S(F) \text{ can be 
 covered by } 2^n \text{ balls 
 of radius } \varepsilon \}.
\]
The characterization is as follows.

\begin{theorem} \label{thm:entropy}
For all $n\in\N$ and $0< \delta<1$, we have 
\[
\e_{n k_\delta}(S) 
\;\le\; e_n^\arb(S,\delta) 
\;\le\; \e_{n k_\delta'}(S) 
\]
where 
$k_\delta=\lceil\log_2(1/\delta)\rceil$ and
$k_\delta'=\lceil\log_2(1/\delta+1)-1\rceil.$
\end{theorem}

\medskip
\goodbreak

For continuous measurements, we get the following bound. 

\begin{theorem} \label{thm:cont}
If the mapping $S$ is continuous then for all $n\in\N$ and 
$0\le \delta<1$, we have 
$$ e^{\cont}_n(S,\delta) \,\le\, \e_n(S). $$
\end{theorem}

\medskip
Combining this with Theorem~\ref{thm:entropy} and 
$e^{\arb}_n(S,\delta)\le e^{\cont}_n(S,\delta)$, 
we obtain upper and lower estimates for the minimal errors for continuous measurements in terms of entropy numbers, that are quite sharp in case of `large' noise $\delta\ge\delta_0>0$; namely
\begin{equation}\label{eq:cont-intro}
 \e_{bn}(S) 
 \;\le\; e_n^\cont(S,\delta) \;\le\; \e_n(S),
\end{equation}
where $b=\lceil\log_2(1/\delta_0)\rceil$.
We do not fully understand the behavior 
of $e_n^\cont(S,\delta)$ for `small' noise, but we do know that there are nontrivial problems $S$ such that, even for arbitrarily small $\delta>0,$ continuous information has the same power as discontinuous information, see Example~\ref{example1}.
%
%
%

\medskip
We illustrate the obtained results by the example 
$$S=I_{p,q}^m\co B^m_p\to\ell^m_q$$ 
with $I_{p,q}^m(x)=x$ for $x\in\R^m$, being the identity on $\R^m$ and $B^m_p$ the unit ball of $\ell_p^m$.
That is, we want to approximate vectors 
from $B^m_p$ in the norm of $\ell^m_q$. 
Let us first collect the bounds that we have 
for this example in the different settings.

For noise-free information ($\delta=0$), 
it has been observed in~\cite{KNU25} that 
\[
e^\cont_{n}(I_{p,q}^m,0) \;=\; 0 
\qquad\text{ for }\quad n > \lceil \log_2(m+1)\rceil. 
\]
This bound on 
the amount $n$ of necessary information does 
not depend on $p$,  $q$ and 
the desired error bound $\e$. 
For linear  or nonadaptive continuous information and $p=q$
the error is 1 for all $n < m$, hence 
one cannot improve the initial error ($n=0$) at all. 
This follows from the Borsuk-Ulam theorem, 
see e.g.~\cite{KNU25}. 
%



In addition, for 
information with noise bounded by $\delta$, 
it is not difficult to see that
\[
e^\lin_n(I_{p,q}^m,\delta) \;\ge\; \delta\,\|I_{p,q}^m\|
\qquad\text{ for all }\; n\ge 1, 
\]
where $\|I_{p,q}^m\|=\max\{\|x\|_q\mid \|x\|_p=1\}$ 
is the norm of the embedding $I_{p,q}^m,$
since for any $x\in B_p^m$, the inputs $\delta x$ and $-\delta x$ 
can lead to the same information $\mathbf y=0$.


Concerning continuous and more general information, let us first note that the entropy numbers of $I_{p,q}^m$ satisfy  
\[
\e_n(I^m_{p,q}) \;\asymp\; \left(\frac{\log(m/n+1)}
{n}\right)^{\frac{1}{p}-\frac{1}{q}}
\]
for $\log(m)\le n \le m$ and $p\le q$, 
and 
$\e_n(I^m_{p,p})\asymp 2^{-n/m}$ for all $n\in\N$, 
%
see~\cite{kuehn01,sch84,VybKos},
with the hidden constants possibly depending on~$p$ 
and~$q$, but not on~$m$ and~$n$.
(We write `$a_n\asymp b_n$ for $n\in\mathbb{I}\subset\N$' if $c\cdot b_n\le a_n\le C\cdot b_n$ for some $c,C>0$ (aka the hidden constants) and all $n\in\mathbb{I}$; similarly for the one-sided inequalities $\lesssim$ and $\gtrsim$.)
%
From Theorem~\ref{thm:entropy} and~\eqref{eq:cont-intro} 
we obtain, for all $\delta \in (0,1)$, 
the upper and lower bound 
\[
c_\delta \left(\frac{\log(m/n+1)}
{n}\right)^{\frac{1}{p}-\frac{1}{q}}
\;\le\;
e_n^{\cont}(I_{p,q}^m,\delta) 
\;\le\;
C \left(\frac{\log(m/n+1)}
{n}\right)^{\frac{1}{p}-\frac{1}{q}}
\] 
for $\log(m) \le n \le c_\delta m$ and $p< q$, 
and some $c_\delta\gtrsim\log(1/\delta)^{-1}$
that possibly depends on $p,q$.
From~\eqref{eq:cont-intro} and \cite[12.1.13]{Pietsch-ideals}, 
we also obtain 
$({\delta}/{2})^{n/m} \le e_n^{\cont}(I_{p,p}^m,\delta) \le 4\cdot (1/2)^{n/m}$ for all $n,m\in\N$ and $1\le p\le \infty$. 
(Mind the index shift in our definition of $\eps_n$.)

\medskip

We now interpret the bounds and compare the different settings. 
We think it is 
interesting to compare continuous 
(or more general) information with linear information. 
We do it here only for the cases 
$(p,q)=(2,\infty)$ and 
$(p,q)=(1,2)$ where the results are quite different.  

For linear information 
we have  
$e_n^\lin(I^m_{2,\infty},0) \asymp 1$
for  $2n\le m$, 
see~\cite{vyb08},
while 
the above implies 
$e_n^\cont(I^m_{2,\infty},\delta) <\eps$ 
for $n\asymp \log(m)\,\e^{-2}$.
This shows that noisy continuous information 
(for any $\delta<1$) is much more powerful 
than noise-free linear information
for uniform approximation on the Euclidean ball. 

In contrast, we have 
$e_n^\lin(I^m_{1,2},0)\asymp \e_n(I^m_{1,2})$
for $2n \le m$, see again~\cite{vyb08}. 
Therefore, noise-free linear information is of the same 
power as arbitrary (discontinuous) noisy information, 
as long as $2n \le m$ and the noise is not too small. 
To a certain extent this is even true if the 
linear information is noisy since 
\begin{equation*} 
e_n^\lin(I^m_{1,2},\delta)\asymp 
e_n^\lin(I^m_{1,2}, 0) + \delta
\end{equation*}
for all $n\in\N$ follows from work in compressed sensing. Here we can almost use Foucart~\cite{Fou22},  
Theorem 14.6 on page 121. 
We need to change the normalization (see 
also Remark 14.7 for yet another 
normalization) and actually we need to use 
``Rademacher'' instead of ``Gaussians'', 
since we need that the measurement matrix $A$ 
satisfies
$|a_{i,j}| \le 1$ for each entry to guarantee 
``if $\Vert x \Vert_1 \le 1$ then 
$ \Vert Ax\Vert_\infty \le 1$'' which is our condition 
on the information functionals, i.e., 
$\Vert \lambda_i \Vert \le  1$. 
See also Foucart and Rauhut~\cite{FR13},  
Theorem 9.13.
So for instance, if we want to achieve an error of order $\varepsilon$
and the noise is of order $\varepsilon^{100} \lesssim \delta \lesssim \varepsilon$,
then the required amount of linear information has the order $\log(m) \varepsilon^{-2}$
while the required amount of continuous information has the order at least $\log(m)\, \varepsilon^{-2} (\log 1/\varepsilon)^{-2}$.
This shows that noisy continuous information 
is roughly as powerful
as noise-free linear information for Euclidean approximation on the $\ell_1^m$-ball.

\medskip
Theorems~\ref{thm:entropy} and \ref{thm:cont} are proven respectively in Sections~\ref{sec:entropy} and \ref{sec:ub-cont}, where also corresponding bounds on Lipschitz measurements are presented in Theorem~\ref{thm:lip}. The last Section~\ref{sec:diago} is devoted to diagonal operators.

\bigskip
\goodbreak

\section{Arbitrary measurements}\label{sec:entropy}

Here we prove Theorem~\ref{thm:entropy}.

\begin{proof}
First we show the upper bound on $e^\arb_n(S,\delta)$. 
Fix $\eta>0$ and let
$G_n\subset G$ be a set such that $\#G_n=2^{n k_\delta'}$ and 
$$\min_{g\in G_n}\dG(S(f),g) < \e_{n k_\delta'}(S) + \eta\quad\mbox{for all}\quad f\in F.$$ 
We number the elements as 
$$G_n=\{g_{i_1,\ldots,i_n}:\;1\le i_j\le 2^{k_\delta'},\;1\le j\le n\},$$
and decompose $F$ into $2^{n k_\delta'}$ pairwise disjoint sets $(F_{i_1,\ldots,i_n})_{1\le i_j\le 2^{k_\delta'}}$ 
such that
$$
 f \in F_{i_1,\ldots,i_n}
 \quad
 \Longrightarrow
 \quad
 \dG(S(f),g_{i_1,\ldots,i_n})=\min_{g\in G_n}\dG(S(f),g).
$$
Then we construct an algorithm $A_n$ (based on arbitrary noisy information) as follows. 
For $f\in F_{i_1,\ldots,i_n}$ we set 
$$\lambda_j(f)=v_{i_j},\quad\mbox{where}\quad v_{i_j}=-1+2\,\frac{i_j-1}{2^{k_\delta'}-1},\quad 1\le j\le n.$$ 
Since $\delta<1/(2^{k_\delta'}-1)$ the $\delta$-neighborhoods of $v_{i_j}$ are pairwise disjoint. Hence, for each $\mathbf y$ satisfying $|y_j-\lambda_j(f)|\le\delta,$ $1\le j\le n,$ we can determine the correct value of $\lambda_j(f)$ for all $j\le n$, and hence we can determine $i_1,\hdots,i_n$ with $f\in F_{i_1,\ldots,i_n}$. The approximation $A_n(f)=\varphi(\mathbf y)$ defined by $\varphi(\mathbf y)=g_{i_1,\ldots,i_n}$ satisfies
\[
e(A_n,\delta) \;\le\; \e_{n k_\delta'}(S) + \eta.
\]
Taking $\eta\to0$ implies the result.

\smallskip
To show the lower bound on $e_n^\arb(S,\delta)$, we let $A_n$ be any algorithm that uses a reconstruction map $\varphi$ and $n$ adaptive functionals $\lambda_j$. Define $w_i:=-1+ (2i-1) 2^{-k_\delta}$ with $1\le i\le 2^{k_\delta}$ and 
\[
G_n = \left\{\phi(w_{i_1},\dots,w_{i_n})\co 1\le i_j\le 2^{k_\delta}, 1\le j\le n \right\}\subset G.
\]
Now, for $f\in F$, we take 
$w_{i_1}$ with $i_1\in\{1,\dots,2^{k_\delta}\}$ such that 
$$w_{i_1}-2^{-k_\delta}\le \lambda_1(f)<w_{i_1}+2^{-k_\delta}$$
(if $\lambda_1(f)=1$ then $i_1=2^{k_\delta}$). By induction, for $j=2,3,\ldots,n$, we take $w_{i_j}$ such that 
\[
w_{i_j}-2^{-k_\delta}\le \lambda_j(f;w_{i_1},\ldots,w_{i_{j-1}})<w_{i_j}+2^{-k_\delta}
\]
(if $\lambda_j(f;w_{i_1},\ldots,w_{i_{j-1}})=1$ then $i_j=2^{k_\delta}$).
Since $2^{-k_\delta}\le \delta$, 
we observe that $\mathbf y=(w_{i_1},\ldots,w_{i_n})$ is noisy information about $f,$ which implies 
\[ e(A_n,\delta)\,\ge\,d_G(S(f),\phi(\mathbf y))\,\ge\,\min_{g\in G_n} d_G(S(f),g).
\]
Since $f$ is arbitrary and $\#G_n\le 2^{nk_\delta}$, this completes the proof. 
\end{proof}

\begin{remark}
Observe that the information constructed in the proof of the upper bound is nonadaptive. Hence we also showed that if arbitrary functionals are allowed then adaptive information is of the same power as nonadaptive information. While this claim is trivial in the noise-free case $\delta=0$ (then any adaptive information is also nonadaptive), it is not so obvious in the presence of noise $\delta>0,$ since then the class of nonadaptive information is a proper subset of the class of adaptive information.
\end{remark}

\section{
Continuous measurements}
\label{sec:ub-cont}

We now 
deal with continuous measurements. For that we have to assume that $F$ is equipped with a metric $\dF.$ 
Here, unlike in the case of arbitrary functionals, the use of adaption is crucial. 
\medskip

We start with the proof of Theorem~\ref{thm:cont}.

\begin{proof}
Let $\delta<\delta^+\le 1$.
Let $B_1,\hdots,B_N$ with $N=2^n$ be closed balls of radius $r>\e_n(S)$ that cover $S(F).$
Choose arbitrary $\eta>0,$ and for any $I\subset\{1,2,\ldots,2^n\}$ with $\#I=2^{n-1}$ 
consider the continuous functional
\begin{equation}\label{eq:f} 
\lambda(f)=\frac{2}{\eta}\min\bigg\{\dG\bigg(S(f)\,,\,\bigcup_{i\in I} B_i \bigg),\delta^+\eta\bigg\} -1\in [-1,1].
\end{equation}
Suppose $y$ comes from noisy measurement of $\lambda(f),$ i.e., $|y-\lambda(f)|\le\delta.$ If $y>-1+\delta$ then $\lambda(f)>-1$, which implies that $S(f)$ is not in the balls $B_i$, $i\in I$, i.e., in one of the remaining balls. On the other hand, if $y\le -1+\delta$, we must have $\lambda(f)\le -1 + 2\delta$, meaning that the $\min$ is at most $\delta\eta.$ This, in view of $\delta<\delta^+,$ 
implies that the distance of $S(f)$ and the balls is at most $\delta\eta.$ Hence, if we increase the radius of the balls $B_i$, $i\in I$, by $\delta\eta$, 
we know that $S(f)$ is in one of these balls.

Now we can do $n$ bisection steps, where we possibly increase the radius of the balls in each step by $\delta\eta.$
We end up with a single ball of radius at most $r + n\delta\eta$
that contains $S(f)$.
The output of the algorithm shall be the center of this ball, and the error is at most $r+n\delta\eta.$ 
Since $\eta>0$ can be chosen arbitrarily small
and $r$ arbitrarily close to $\varepsilon_n(S)$, 
we have $e_n^\cont(S,\delta)\le\varepsilon_n(S).$
\end{proof}

\smallskip
As discussed in Section~\ref{sec:main}, the bound $e_n^\cont(S,\delta) \le \varepsilon_n(S)$ is (almost) sharp for `large' noise $\delta\ge\delta_0>0.$
For `small' noise, on the other hand, the bound is generally not sharp. Below we give an example where $e_n^\cont(S,\delta) \asymp e_n^\arb(S,\delta)$ for any $\delta \in (0,1)$ and hence 
$e_n^\cont(S,\delta)$ is much smaller than $\varepsilon_n(S)$
for `small' $\delta$. 

\smallskip
\begin{example}\label{example1}
Consider the approximation of the embedding $S=I^m_{\infty,\infty}: B_\infty^m\to\ell_\infty^m$ using 
continuous measurements. Then we can proceed as follows. Let $x=(x_1,x_2,\ldots,x_m)\in B_\infty^m.$ 
For each $i$th coordinate we do $r$ adaptive measurements: $\lambda_{i,1}(x)=x_i$ 
and, for $j=2,3,\ldots,r$, 
\begin{equation}\label{lij}
\lambda_{i,j}(x)=\left\{\begin{array}{rl}
     \delta^{1-j}(x_i-y_{i,j-1}), &\quad \text{if}\quad |x_i-y_{i,j-1}|\le\delta^{j-1}, \\
      \mathrm{sgn}(x_i-y_{i,j-1}), &\quad\mbox{otherwise},\end{array}\right.
\end{equation}
where $y_{i,j-1}$ is the noisy information about $x$ 
obtained from the $(j-1)$st measurement of $x_i.$ Then 
the total number of measurements equals $n=r\,m,$ and 
all the vectors $x$ that are indistinguishable with respect to 
information $$\mathbf y=(y_{1,1},\ldots,y_{1,r},\ldots,y_{m,1},
\ldots,y_{m,r})$$ are contained in a ball 
of radius $\delta^r$ (with respect to the $\infty$-norm). Taking as $\varphi(\mathbf y)$ the center of this ball
we obtain an algorithm with error $\delta^{n/m}.$

On the other hand, for $n=r\,m$ we have $\varepsilon_n(I_{\infty,\infty}^m)=2^{-n/m},$ which mans that 
$$e_n^\cont(I^m_{\infty,\infty},\delta)\le\varepsilon_{nk_\delta''}(I^m_{\infty,\infty})
\quad\mbox{with}\quad k_\delta''=\lceil\log_2(1/\delta)-1\rceil.$$
Hence, continuous information is as powerful as arbitrary information.
\end{example}

\begin{remark}[Noise-correction]
A possible interpretation of the technique used in Example~\ref{example1} is 
that any measurement $\lambda\in\Lambda^\mathrm{con}$, which is available a priori only up to noise $\delta$, can be estimated to arbitrary precision by using some other noisy continuous measurements. Hence, 
continuous measurements can be used for \emph{noise-correction} below the original noise level. 
The same is not possible with linear measurements.

To be precise, using the measurements~\eqref{lij} with $x_i$ replaced by $\lambda(f)$, we see that a sequence of $r$ noisy continuous measurements (with noise~$\delta$), can be used to estimate $\lambda(f)$ up to error $\delta^r$. That is, with $r\ge \log(\eps)/\log(\delta)$, 
we obtain $y\in\R$ with $|y-\lambda(f)|\le\eps$.
\end{remark}

\medskip
We now focus on Lipschitz continuous measurements.
Recall that the modulus of continuity of $S$ is defined as
$$ \omega_S(\gamma)=\sup\left\{\dG(S(f),S(g))\mid\, \dF(f,g)\le\gamma\right\}.$$
We will also need a `modified' modulus
$$ \widetilde\omega_S(\gamma) = \gamma \cdot \sup\left\{\frac{\dG(S(f),S(g))}{\dF(f,g)} \mid\, \dF(f,g)\le\gamma\right\}.$$
Note that
$\omega_S(\gamma)\le\widetilde\omega_S(\gamma)\le\gamma\,\mathrm{lip}(S),$
where $\mathrm{lip}(S)$ is the Lipschitz constant for $S.$

\begin{theorem} \label{thm:lip}
If the mapping $S$ is uniformly continuous then for all $n\in\N$ and $\delta<1,$ we have
\begin{equation}\label{eq:l} \frac12\,\omega_S\left(\frac{2\delta}{L}\right)\,\le\,e^{\Lip}_{n,L}(S,\delta) \,\le\, \e_n(S)\,+\,n\cdot\widetilde \omega_S\!\left(\frac{2\delta}{L}\right). 
\end{equation}
Otherwise for all $n\in\N$ and $\delta>0$ we have $$e_{n,L}^\Lip(S,\delta)\ge c$$ 
with $c=\lim_{\gamma\to 0^+}\omega_S(\gamma)/2>0.$ 
\end{theorem}  

\begin{proof}
To show the upper bound in \eqref{eq:l} we refer to the proof of Theorem~\ref{thm:cont}. Observe that the functionals $\lambda$ given by \eqref{eq:f} satisfy
\[
 |\lambda(f)-\lambda(g)| \,\le\, \min\left\{ \frac{2}{\eta} d_G(S(f),S(g)),\, 2 \delta^+ \right\}.
\]
If we take $\eta=\widetilde\omega_S(2\delta^+/L)/\delta^+$, the first term in the min is bounded by $L\,d_F(f,g)$ in the case $d_F(f,g) \le 2\delta^+/L$; in the other case the second term is bounded by $L\,d_F(f,g)$. So $\lambda$ has a Lipschitz constant at most $L$ and we obtain the upper error bound 
$$e_{n,L}^{\Lip}(S,\delta)\,\le\,\e_n(S)+n\delta\eta=\e_n(S)+n\bigg(\frac{\delta}{\delta^+}\bigg)\widetilde\omega\bigg(\frac{2\delta^+}{L}\bigg).$$ 
Uniform continuity of $S$ implies continuity of $\tilde\omega_S.$ Letting $\delta^+\searrow\delta$ we get the desired bound.




\smallskip
We now show the lower bound in \eqref{eq:l}.
Let $\gamma=2\delta/L.$ For $\epsilon>0,$ let $f,g\in F$ such that $\dF(f,g)\le\gamma$ 
and $\dG\big(S(f),S(g)\big)\ge\omega_S(\gamma)-\epsilon.$ 
Then for any functional $\lambda$ we have  
$$|\lambda(f)-\lambda(g)|\le L\,\dF(f,g)\le 2\delta,$$ 
which means that $f$ and $g$ are indistinguishable with respect to noisy information $y=(\lambda(f)+\lambda(g))/2.$ Therefore the error of any approximation is at least
$$\dG\big(S(f),S(g)\big)/2\ge\omega_S(\gamma)/2-\epsilon/2,$$
as claimed, since $\epsilon$ can be arbitrarily small.

\smallskip
To show the remaining part of the theorem it is enough to use the known property that the lack of uniform continuity of $S$ implies that the modulus is not continuous at zero.
\end{proof}

\smallskip
While it is not clear whether the upper bound of Theorem~\ref{thm:lip} is sharp, the lower bound in \eqref{eq:l} cannot be improved in general. Indeed, consider the problem of Example~\ref{example1}. Let $L=\delta^{1-r}.$ Then the functionals \eqref{lij} are allowed and we have an algorithm with error $\delta^r=\delta/L=\omega_S(2\delta/L)/2.$ 

\begin{remark}
Since in the proof of the lower bound of Theorem~\ref{thm:lip} we did not use the assumption \eqref{eq:funals}, this bound remains valid if the range of the information functionals is extended from $[-1,1]$ to $\mathbb R.$ Then the error $\omega_S(2\delta/L)/2$ for the problem of Example~\ref{example1} can be attained using just $m$ (instead of $r\,m$) linear measurements $\lambda_i(x)=Lx_i,$ $1\le i\le m,$ and $A_m(x)=\mathbf y/L,$ for noisy information $\mathbf y$ about $x.$
\end{remark}

\begin{remark}
Let us also compare with the work 
on \emph{stable Lipschitz widths} from~\cite{CDPW}. 
In our language, 
they show that for every algorithm based on (arbitrary) noisy information, there is a non-adaptive Lipschitz-continuous algorithm that is not so much worse, at least if we approximate the identity $S\colon G\to G$ on 
a compact subset $F\subset G$ of a Banach space $G$. In fact,  
we obtain from Theorem~5.1 of~\cite{CDPW}
that there is an algorithm $A=\phi\circ N$, 
with 
$N\colon G\to\ell_\infty^n$ and 
$\phi\colon \ell_\infty^n\to G$, 
both being Lipschitz with ``constant'' $C n^{5/4}$, such that 
\begin{equation*}
\sup_{f\in F} \|f - \phi\circ N(f)\|_G 
\,\lesssim\, n^{5/2}\,\e_{cn}(F). 
\end{equation*}
By \emph{rescaling} the above algorithm we can assume that $\phi$ has Lip-constant $C n^{5/2}$, and $N$ has Lip-constant $\asymp 1$. 

If $\gamma>0$ is now the Lipschitz constant of $\phi$, 
then we also have 
\begin{multline*}
\sup_{f\in F}\sup_{z\in\R^n\colon \|z\|_\infty\le\delta} \|f - \phi\left( N(f)+z\right)\|_G \\
\,\lesssim\, n^{5/2}\,\e_{cn}(F) + \gamma \delta 
\,\lesssim\, n^{5/2}\,(\e_{cn}(F) + \delta). 
\end{multline*}
This implies 
\[
e_n^{\cont{\rm -non}}(S,\delta) \,\lesssim\, n^{5/2}\,(\e_{cn}(F) + \delta), 
\]
where $e_n^{\cont-{\rm non}}$ is the $n$th minimal error of non-adaptive algorithms based on continuous measurements. With Theorem~\ref{thm:entropy}, we arrive at
\begin{equation*}
    e_{c k_\delta \cdot n}^{\cont-{\rm non}}(S,\delta) \,\lesssim\, (k_\delta\, n)^{5/2}\,(e_n^\arb(S,\delta) + \delta)
\end{equation*}
with a (different) absolute constant $c>0$ and $k_\delta =\lceil\log(1/\delta)\rceil$.
\end{remark}

\section{Diagonal operators}\label{sec:diago}

For $1\le p\le\infty$ and $\sigma=(\sigma_1,\sigma_2,\ldots)$ with 
$$\sigma_1\ge\sigma_2\ge\sigma_3\ge\cdots\ge 0, 
$$ 
we consider a diagonal operator $D_\sigma:B_p\to\ell_p$ defined by $$D_\sigma(x_1,x_2,x_3,\ldots)=(\sigma_1x_1,\sigma_2x_2,\sigma_2x_3,\ldots),$$
where $B_p$ is the unit ball of $\ell_p$. For the entropy numbers, and all $p$, we have the formula
$$\varepsilon_n(D_\sigma)\;=\;\kappa_n\,\sup_{k\ge 1}\,2^{-n/k}(\sigma_1\sigma_2\cdots\sigma_k)^{1/k}$$
with $\kappa_n\in[1,6],$ see \cite[Prop. 1.7]{GKS87} and \cite{kuehn05}. In particular, if $\sigma_j=j^{-s}$ then $\varepsilon_n(D_\sigma)\asymp\sigma_n=n^{-s},$ where the factor in the '$\asymp$' notation depends on $s.$ 

\medskip
Now, we deal with the $n$th minimal error of measurement-based approximations of $x=(x_1,x_2,\ldots)$ in the unit ball of $\ell_p.$ 

\subsection
{Exact measurements, $\delta=0$} 

It is well known that in the case of linear measurements and any $p$, the $n$th optimal measurements are $\lambda_i(x)=x_i,$ $1\le i\le n,$ and the optimal approximation is $A_n(x)=(\sigma_1x_1,\sigma_2x_2,\ldots,\sigma_nx_n,0,0,0,\ldots).$ This gives the $n$th minimal error $$e_n^\mathrm{lin}(D_\sigma,0)_p=e(A_n,0)_p=\sigma_{n+1}.$$
(We add the subscript $p$ to indicate the space we are in.)

\smallskip
For continuous measurements we can apply the algorithm from \cite{KNU25} to recover from $x$ its coefficients $x_1,x_2,\ldots,x_m$ with arbitrarily small error using $\lceil\log_2(m+1)\rceil+1$ (adaptive) measurements. Hence, with $n$ measurements one can recover the coefficients $x_i,$ $1\le i\le m=2^{n-2}-1$ with arbitrarily small error. This gives
$$e_n^\cont(D_\sigma,0)_p\le\sigma_{2^{n-2}}.$$
In particular, for $\sigma_j=j^{-s}$ we have an exponential speed-up, since $e_n^\cont(D_\sigma,0)_p\le 2^{-s(n-2)}.$ The lower bound on $e_n^\cont(D_\sigma,0)_p$ is unclear. To complete the picture for exact measurements we also mention that $e_n^\mathrm{arb}(D_\sigma,0)=0.$

\subsection{Noisy measurements, $0<\delta<1$} 
For linear measurements $\lambda$ we assume that $||\lambda\|_p\le 1,$ cf. Remark~\ref{rem0}. Consider approximations $\widetilde A_n$ similar to those from the linear case and exact measurements above, i.e., $\widetilde A_n(x)=(\sigma_1y_1,\sigma_2y_2,\ldots,\sigma_ny_n,0,0,0\ldots),$ where $|y_i-x_i|\le\delta,$ $1\le i\le n.$ Then for $1\le p<\infty$ we have
\begin{eqnarray}\nonumber
e(\widetilde A_n,\delta)_p &=& \bigg(\sup_{\|x\|_p\le 1}\,\sup_{|y_i-x_i|\le\delta}\sum_{i=1}^n\sigma_i^p|y_i-x_i|^p+\sum_{j=n+1}^\infty\sigma_j^p|x_j|^p\bigg)^{1/p}\\
&=&\bigg(\delta^p\sum_{i=1}^n\sigma_i^p+\sigma_{n+1}^p\bigg)^{1/p}, \label{ptherr}
\end{eqnarray}
while for $p=\infty$ we have $e(\widetilde A_n,\delta)_\infty=\max(\delta\sigma_1,\sigma_{n+1}).$ 

For a lower bound we have that for any $p$ the error of any approximation using $n$ measurements is at least $\max(\delta\sigma_1,\sigma_{n+1}).$ Indeed, since for any $x$ with $\|x\|_p\le\delta$ and any $\lambda$ it holds that $|\lambda x|\le\delta,$ all elements of the ball of radius $\delta$ and center at $0$ are indistiguishable with respect to zero information. Hence, on one hand, the error of any approximation is lower bounded by
$$\sup\{\|D_\sigma x\|_p\,\mid\;\|x\|_p\le\delta\}=\delta\sigma_1,$$
and, on the other hand, it is not smaller than the minimal error $\sigma_{n+1}$ from exact measurements. This means that
$$e_n^\mathrm{lin}(D_\sigma,\delta)_\infty=\max(\delta\sigma_1,\sigma_{n+1}).$$ 
Unfortunately, the approximation $\widetilde A_n$ is $n$th optimal only for $p=\infty.$ What is the exact value of $e_n^\mathrm{lin}(D_\sigma,\delta)_p$ for $p<\infty$ is an open problem.

\begin{remark}
Suppose $p=2.$ Then the $n$th minimal error is known exactly in very special cases only. One of these is when $\sigma_{n+1}=0$ and the worst case error is taken over the whole space $\ell_2$ (instead of the unit ball). Then \eqref{ptherr} is valid and the $n$th minimal error equals 
$\delta\sqrt{\sum_{i=1}^n\sigma_i^2},$ see the Appendix of \cite{PS23}.

It is also worthwhile to mention that much more is known if we assume that the noise is bounded in the $\ell_2$-norm (instead of $\ell_\infty$), i.e., $\sum_{i=1}^n|y_i-\lambda_ix|^2\le\delta^2.$ Then the corresponding $n$th minimal error equals
$$\widehat e_n^{\,\lin}(D_\sigma,\delta)_2=\sqrt{\sigma_{n+1}^2+\frac{\delta^2}{n}\sum_{j=1}^n(\sigma_j^2-\sigma_{n+1}^2)},$$
see \cite[p.79]{Pl96}. Obviously $\widehat e_n^{\,\lin}(D_\sigma,\delta)_2\le e_n^{\lin}(D_\sigma,\delta)_2\le \widehat e_n^{\,\lin}(D_\sigma,\delta\sqrt{n})_2.$ 
\end{remark}

\smallskip
We switch to continuous measurements. In this case we have the bounds of Theorems~\ref{thm:entropy} and \ref{thm:cont},
\begin{equation}\label{conteq}
\varepsilon_{n\lceil\log_2(1/\delta)\rceil}(D_\sigma)\le e_n^\cont(D_\sigma,\delta)_p\le\varepsilon_n(D_\sigma).
\end{equation}
Let's see what we can obtain if we apply the method of Example~\ref{example1} relying on reducing uncertainties of the successive coordinates. In what follows we assume $p=\infty$ and $\delta<1/2.$ Then, to get an approximation with error at most $\varepsilon$, it is enough to reduce the `uncertainty' of $\sigma_ix_i$ to~$\varepsilon,$ for all $i\le m=\min\{k:\,\sigma_{k+1}\le\varepsilon\}.$ This can be done with the help of $$n_i=\bigg\lceil\frac{\ln(\sigma_i/\varepsilon)}{\ln(1/\delta)}\bigg\rceil$$ measurements. Hence, in total, we have error $\varepsilon$ at cost  
\begin{equation}\label{ncost}
n=\sum_{i=1}^m n_i=\frac{\ln(\frac{\sigma_1\sigma_2\cdots\sigma_m}{\varepsilon^m})}{\ln(1/\delta)}+r\qquad\mbox{with}\quad r\in[0,m].
\end{equation}
If the problem is finite dimensional, i.e., $\sigma_m>\sigma_{m+1}=0$ for some $m$ (which is the case of Example~\ref{example1}, where $\sigma_k=1$ for all $1\le k\le m$), then 
$$n\le \frac{m\log_2(1/\varepsilon)+\log_2(\sigma_1\sigma_2\cdots\sigma_m)}{\log_2(1/\delta)}+m\approx m\,\frac{\log_2(1/\varepsilon)}{\log_2(1/\delta)}\quad (\mbox{as}\;\varepsilon\to 0).$$
(Here `$\approx$' means the asymptotic equality.) Thus using $n$ continuous measurements one can obtain an approximation with error asymptotically equal to at most $\delta^{n/m},$ and therefore the lower bound in \eqref{conteq} is achieved, i.e.,
$$e_n^\mathrm{con}(D_\sigma,\delta)_\infty\approx \delta^{n/m}=2^{-\frac{n\log_2(1/\delta)}{m}}\asymp\varepsilon_{n\lceil\log_2(1/\delta)\rceil}(D_\sigma)\qquad (\mbox{as}\;n\to\infty).$$

However, an analogous result does not hold when all $\sigma_k$'s are positive. Indeed, as all $n_i$'s in \eqref{ncost} are positive integers then applying our method with $n$ measurements one can obtain the error not smaller than $\sigma_{n+1}\asymp\varepsilon_n(D_\sigma).$
 
Take as an example $\sigma_k=k^{-s}$ and, for simplicity, $\varepsilon=\sigma_{m+1}.$ Then, by Stirling's formula,
$$\frac{\sigma_1\sigma_2\cdots\sigma_m}{\sigma_{m+1}^m}=\bigg(\frac{(m+1)!}{(m+1)^{m+1}}\bigg)^{-s}\approx\bigg(\frac{\mathrm e^{m+1}}{\sqrt{2\pi(m+1)}}\bigg)^s\quad (\mbox{as}\;m\to\infty)$$ and
$$\frac{\ln(\frac{\sigma_1\sigma_2\cdots\sigma_m}{\varepsilon^m})}{\ln(1/\delta)}\approx s\log_2\!\mathrm e\bigg(\frac{m+1}{\log_2(1/\delta)}\bigg).$$
This would give us the upper bound on $e_n^\mathrm{con}(D_\sigma,\delta)_\infty$ proportional to $\big(n\log_2(1/\delta)\big)^{-s}$ if it was not for the component $r$ in \eqref{ncost}. The best we can get is only
$$n\lessapprox m\bigg(1+s\frac{\log_2\!\mathrm e}{\log_2(1/\delta)}\bigg)\le (1+s\log_2\!\mathrm e)\,\varepsilon^{-1/s}$$
and, as a consequence, 
$$e_n^\cont(D_\sigma,\delta)_\infty\lesssim \varepsilon_n(D_\sigma),$$
as we already knew from~\eqref{conteq}.

This is only an upper bound, the actual behavior of $e_n^\cont(D_\sigma,\delta)$ for infinite dimensional $D_\sigma$ is unknown.

\smallskip
For arbitrary measurements we have the formula for $e_n^\mathrm{arb}(D_\sigma,\delta)_p$ given by Theorem~2. In particular, for $\sigma_k=k^{-s}$ we have 
$$e_n^\mathrm{arb}(D_\sigma,\delta)_p\asymp\big(n\lceil\log_2(1/\delta)\rceil\big)^{-s}.$$

\linespread{0.9}



\newpage

\noindent
\address{D.K., 
University of Passau; 
\texttt{david.krieg@uni-passau.de};  \\
E.N., 
Friedrich Schiller University Jena; 
\texttt{erich.novak@uni-jena.de}; \\
L.P., 
University of Warsaw;
\texttt{leszekp@mimuw.edu.pl};\\
M.U., 
Johannes Kepler University Linz;
\texttt{mario.ullrich@jku.at}
}

\end{document}